\documentclass[gen-j-l,pamsfonts]{amsart}
\usepackage{amssymb,amsmath,amsfonts}

\usepackage{graphicx}   

\newtheorem{theorem}{Theorem}

\newtheorem{proposition}{Proposition}
\newtheorem{question}{Question}

\begin{document}

\title[Brieskorn manifolds, generalised Sieradski groups and coverings]{Brieskorn manifolds, generalised Sieradski groups, and coverings of lens spaces}

\author{Tatyana Kozlovskaya}
\address{Magadan Institute of Economics, 685000, Magadan, Russia}
\email{konus$\_$magadan@mail.ru}

\author{Andrei Vesnin}
\address{Sobolev Institute of Mathematics, 630090, Novosibirsk, Russia}
\email{vesnin@math.nsc.ru}

\thanks{The research was done under partial financial support by RFBR (grant 15-01-07906).}



\dedicatory{To Sergei Matveev on the occasion of his seventieth birthday}

\maketitle

\begin{abstract}
The Brieskorn manifolds $B(p,q,r)$ are the $r$-fold cyclic coverings of the 3-sphere $S^{3}$ branched over the torus knot  $T(p,q)$. The generalised Sieradski groups $S(m,p,q)$ are groups with $m$-cyclic pre\-sen\-tation $G_{m}(w)$, where defining word $w$ has a special form, depending of $p$ and $q$. In particular, $S(m,3,2) = G_{m}(w)$ is the group with $m$ generators $x_{1}, \ldots, x_{m}$ and $m$ defining relations $w(x_{i}, x_{i+1}, x_{i+2})=1$, where  $w(x_{i}, x_{i+1}, x_{i+2}) = x_{i} x_{i+2} x_{i+1}^{-1}$. Presentations of $S(2n,3,2)$ in a certain form  $G_{n}(w)$ were investigated by Howie and Williams. They proved that the $n$-cyclic presentations are  geometric, i.e., correspond to the spines of closed orientable 3-manifolds. We establish an analogous result for the groups $S(2n,5,2)$. It is shown that in both cases the manifolds are $n$-fold cyclic branched coverings of lens spaces. To classify some of constructed manifolds we used the Matveev's computer program ``Recognizer''.  
\end{abstract}

\tableofcontents 

\section{Three-manifolds with cyclic symmetry} \label{sec1} 

\subsection{Examples of manifolds with cyclic symmetry}

In this paper we discuss a relation between 3-dimensional closed orientable Brieskorn manifolds and the generalised Sieradski groups. The Brieskorn manifolds $B(p,q,r)$, firstly introduced in~\cite{Brieskorn}, are the $r$-fold cyclic coverings of the 3-sphere $S^{3}$ branched over the torus knots $T(p,q)$. The generali\-sed Sieradski groups $S(m,p,q)$, introduces in  \cite{CavicchioliHegenbarthKim} for  positive integers $p$ and $q$, where $p = 1 + d q$ and $d \geq 1$, are groups with cyclic presentation $G_{m}(w)$, where defining word $w$ will be described in subsection~\ref{subsection2.3}. In particular, $S(m,3,2) = G_{m}(w)$ is the group with  $m$ generators $x_{1}, \ldots, x_{m}$ and $m$ defining relations $w(x_{i}, x_{i+1}, x_{i+2})=1$, where $w(x_{i}, x_{i+1}, x_{i+2}) = x_{i} x_{i+2} x_{i+1}^{-1}$. 

Cyclic presentations of the groups $S(2n,3,2)$ (the case $q=2$ and $d=1$) in the form $G_{n}(w)$ were investigated by Howie and Williams in~\cite{Howie}. Namely, they proved that $n$-cyclic presentations of $S(2n,3,2) = G_{n} (x_{i} x_{i+2} x_{i+1}^{-1})$ are \emph{geometric}, i.e., correspond to spines of closed orientable 3-manifolds. Here we establish the analogous result for the groups $S(2n,5,2)$ (the case $q=2$ and $d=2$). More exactly, theorem~\ref{teo4} states that $n$-cyclic presentations 
$$
S(2n,5,2) = G_{n} ( x_{i} x_{i+1} x_{i+2}^{2} x_{i+3} x_{i+4} x_{i+3}^{-1} x_{i+2}^{-1} x_{i+1} x_{i+2} x_{i+3} x_{i+2}^{-1} x_{i+1}^{-1}) 
$$ 
are geometric. Propositions~\ref{ass1} and~\ref{ass2} state that in both cases manifolds are $n$-fold cyclic branched coverings of lens spaces.  

For last decades topological and geometrical properties of 3-manifolds are intensively studying from various points of view. An approach to constructing and describing a manifold depends on a problem posted. The approach can be based on a triangulation of a manifold, its Heegaard splitting, surgery description, branched covering description, etc. Any app\-ro\-ach can be useful in a suitable context. For basic definitions and facts of 3-manifold theory we refer to Hempel's book~\cite{Hempel}. One more useful approach is based on presenting 3-manifolds by spines. The theory of spines of 3-manifolds was developed in papers by Matveev, his students and followers. Main results the theory of spines are presented in his book~\cite{MatveevBook}. Methods and results on tabulating of 3-manifolds are also presented in his survey~\cite{MatveevSurvey}. The presentation of 3-manifolds by their spines is realised in the program tools \emph{3-manifold Recognizer}~\cite{Recognizer}, which was created and developed by Matveev and his group. The tools contain the original software for distinguishing and investiga\-ting 3-manifolds as well as the huge database of 3-manifolds, presented by their spines. In particular, the developed software is useful for studying  symmetry groups of 3-manifolds and for understanding covering properties of 3-manifolds in a lot of concrete cases.  

In this paper we will consider connected closed orientable 3-manifolds with cyclic symmetry which acts with fixed points. Moreover, we will interested in cases when the symmetry corresponds to a presentation of the manifold as a cyclic branched covering of the 3-sphere $S^{3}$ or of the lens spaces. Examples of such manifolds are well known in the literature. Let us recall some of them.  
\begin{itemize}
\item The spherical and hyperbolic dodecahedral spaces constructed by Weber and Seifert in 1931 in the paper~\cite{WeberSeifert}: the first is the $3$-fold cyclic covering of $S^{3}$ branched over the trefoil knot, and the second is a $5$-fold cyclic covering of $S^{3}$ branched over the 2-component Whitehead link. 
\item The smallest volume closed orientable hyperbolic 3-manifold, disco\-ve\-r\-ed by Matveev and Fomenko~\cite{MatveevFomenko} and independently by Weeks~\cite{Weeks}, is the $3$-fold cyclic covering of $S^{3}$  branched over the 2-bridge knot~$7/3$~\cite{MednykhVesnin1}. 
\item  The Fibonacci manifolds, constructed by Helling, Kim and Mennicke~\cite{HellingKimMennicke}, are the $n$-fold cyclic coverings of $S^{3}$ branched over the figure-eight knot.  
\item The Sieradski manifolds, defined by Sieradski in~\cite{S} and investigated by Cavicchioli, Hegenbarth and Kim in~\cite{CavicchioliHegenbarthKim}, are the $n$-fold cyclic coverings of $S^{3}$ branched over the trefoil knot.  
\end{itemize}  

As one knows, the trefoil knot belongs to the family of torus knots (it has the corresponding notation $T(3,2)$). Thus, the Sieradski manifolds from~\cite{CavicchioliHegenbarthKim, S} belong to the bigger class of manifolds, known as the Brieskorn manifolds. 

\subsection{Brieskorn manifolds}

Recall that Brieskorn~\cite{Brieskorn}  (see also Milnor's book~\cite{Milnor68}) initiated a study of the following objects. For positive integer $\alpha_{1}, \alpha_{2}, \ldots, \alpha_{n+1} \geq 2$ consider the polynomial    
$$
f(z_{1}, z_{2}, \ldots, z_{n+1}) = (z_{1})^{\alpha_{1}} + (z_{2})^{\alpha_{2}} + \cdots + (z_{n+1})^{\alpha_{n+1}}.  
$$

\emph{A Brieskorn manifold} $\mathcal B(\alpha_{1}, \alpha_{2}, \ldots, \alpha_{n+1})$ is defined as the intersection of the complex hyperplane $V = f^{-1} (0)$ with the $(2n+1)$--dimensional sphere of unit radius    
$$
S^{2n+1} = \{ (z_{1}, z_{2}, \ldots, z_{n+1}) \in \mathbb C^{n+1} \, : \, |z_{1}|^{2} + |z_{2}|^{2} + \cdots + |z_{n+1}|^{2} = 1 \}.
$$

Brieskorn manifolds are smooth manifolds of dimension $2n-1$. The interest to these manifolds is motivated, in particular,  by the following surprising fact, proved by Milnor~~\cite{Milnor68}: Brieskorn manifolds, corresponding to polynomials 
$$
f(z_1, z_2, z_3, z_4, z_5) = z_1^{6k-1} +  z_2^3 + z_3^2 + z_4^2 + z_5^2
$$ 
where $k=1, 2, \ldots, 28$, are $28$ pairwise non-diffeomorphic exotic spheres and each of them is homeomorphic to the usual 7-dimensional sphere. 

We will interested in the case when Brieskorn manifolds are 3-dimensional. Similar to Milnor's notation from~\cite{Milnor75}, for positive integers $p,q,r \geq 2$ we denote  
$$
\mathcal B(p,q,r) = \{ (z_{1}, z_{2}, z_{3}) \in \mathbb C^{3} \, : \, z_{1}^{p} + z_{2}^{q} + z_{3}^{r} = 0, \quad |z_{1}|^{2} + |z_{2}|^{2} + |z_{3}|^{2} = 1 \}.  
$$ 
According to lemma~1.1 from~\cite{Milnor75}, the manifold $\mathcal B(p,q,r)$ is the $r$-fold cyclic covering of the sphere $S^{3}$, branched over the torus link $T(p,q)$. Since parameters $p$, $q$, $r$ can be permuted in the definition of $\mathcal B(p,q,r)$, this manifold is also the $q$-fold cyclic covering of $S^{3}$, branched over $T(r,p)$, as well as the $p$-fold cyclic covering of $S^{3}$, branched over $T(q,r)$. Recall that the torus link $T(p,q)$ can be defined as the set of points $(z_1, z_2)$ on the unit sphere which satisfy the equation $z_1^p + z_2^q = 0$. Such a link has $d$ components, where $d = \gcd (p,q)$. An $n$-th component, $1 \leq n \leq d$, can be parametrized as the following: 
$$
z_1 = \exp(2\pi i t/ p), \qquad z_2 =\exp(2 \pi i (t+n+ 1/2) /q ) 
$$
for $0 \leq t \leq pq/d$.

Another definition of the torus links can be done in terms of braids. Let $B_{p}$ be the group of geometrical braids on $p$ strings with standard generators $\sigma_{1}, \sigma_{2}, \ldots, \sigma_{p-1}$. Then the closure of the braid $(\sigma_{1} \sigma_{2} \cdots \sigma_{p-1})^{q}$ is the torus link $T(p,q)$. 

It was shown in~\cite{Milnor75} that the Brieskorn manifold $\mathcal B(p,q,r)$ is spherical if $1/p + 1/q + 1/r > 1$; nilpotent if $1/p + 1/q + 1/r =1$; and $\widetilde{\operatorname{SL}}(2, \mathbb R)$-manifold if $1/p + 1/q + 1/r < 1$. 

\section{Cyclically presented groups} \label{sec2}

\subsection{Cyclic presentation and defining word} 

Let $\mathbb F_{n}$ be the free group of rank $m \geq 1$ with generators $x_{1}, x_{2}, \ldots, x_{m}$ and let $w = w(x_{1}, x_{2}, \ldots , x_{m})$ be a cyclically reduced word in $\mathbb F_{m}$. Let $\eta : \mathbb F_{m} \to F_{m}$ be an automorphism given by $\eta (x_i) = x_{i+1}$, $i=1,\dots,m-1$, and $\eta (x_m) = x_1$. The presentation 
$$
G_m (w) \, = \, \langle x_1, \dots , x_m \ | \ w=1, \, \eta (w)=1, \, \dots, \, \eta^{m-1} (w)=1 \rangle ,
$$
is called an $m$-\emph{cyclic presentation} with \emph{defining word} $w$. A group $G$ is said to be \emph{cyclically presented group} if $G$ is isomorphic $G_{m}(w)$ for some $m$ and $w$~\cite{Johnson}. Beside purely algebraic questions (finiteness, commensurability, arithmeticity, etc.), the following two questions about cyclically presented groups were posted and investigated by many authors (see \cite{CavicchioliHegenbarthKim, Howie} for example). 

\begin{question}
What cyclically presented groups are isomorphic to the fundamental groups of closed connected orientable 3-manifolds? 
\end{question}

A 2-dimensional subpolyhedron $P$ of a closed connected 3-manifold $M$ is called a \emph{spine} of $M$ if $M \setminus \operatorname{Int} B^{3}$ collapses to $P$, where $B^{3}$ is a 3-ball in $M$ \cite{MatveevBook}. We shall say that a finite presentation $\mathcal P$ corresponds to a spine of a 3-manifold if its presentation complex $K_{\mathcal P}$ is a spine of a 3-manifold. 

\begin{question}
What cyclic presentations $G_{m} (w)$ of groups are geometric, i.e., correspond to the spines of closed connected orientable 3-manifolds? 
\end{question}

Let us recall some examples of cyclically presented groups which are fundamental groups of 3-manifolds, as well as demonstrate  that some families of cyclically presented groups can not be realised as groups of hyperbolic 3-orbifolds (in particular, hyperbolic 3-manifolds) of finite volume. 

\subsection{Fibonacci groups}

We start with groups, corresponding to the defining word  $w(x_{1}, x_{2}, x_{3}) = x_{1} x_{2} x_{3}^{-1}$. Cyclically presented groups  
$$
F(2,m) \ = \ \langle x_1, \dots , x_m \ | \ x_i x_{i+1} \, = \, x_{i+2}, \quad i = 1, \dots , m \rangle, 
$$
where subscripts are taken by $\operatorname{mod} m$, are known as  \emph{the Fibonacci groups}. It was shown by Helling, Kim and Mennicke~\cite{HellingKimMennicke} that if number of generators is even, $m=2n$, $n\geq 4$, then the groups $F(2,2n)$ are the fundamental groups of hyperbolic 3-manifolds. These manifolds were called \emph{the Fibonacci manifolds}. Thus, for even $m \geq 8$ the group $F(2,m)$ is infinite and torsion-free. For odd $m$ situation is different. If $m$ is odd, then $F(2,m)$ has elements of finite order. More exactly, the product $v = x_{1} x_{2} \ldots, x_{m}$ has order two (see  \cite{BardakovVesnin} for details). Therefore, for odd $m$ the group $F(2,m)$ can not be the fundamental group of a hyperbolic 3-manifold. Moreover, it was shown by Maclachlan~\cite{Maclachlan}, that $F(2,m)$ can not be the group of hyperbolic 3-orbifold of finite volume. Recently, Howie and Williams~\cite{Howie} proved that for odd $m \geq 3$ the group $F(2,m)$ is the fundamental group of a 3-manifold if and only if $m=3,5,7$. In all this cases the group is finite cyclic.

A straightforward generalisation of the groups  $F(2,m)$, namely, the groups $F(r,m)$ for $r \ge 2$ and $m \ge 3$ were defined in~\cite{JWW} as follows: 
$$
F(r,m) \, = \, \langle x_1, \dots , x_m \ | \  x_i x_{i+1} \, \dots x_{i+r-1}= \, x_{i+r}, \quad i = 1, \dots , m \rangle ,
$$
where subscripts are taken by $\operatorname{mod} m$. The groups $F(r,m)$ are usually called~\emph{Fi\-bo\-nac\-ci groups} too. A survey of results about finiteness of these groups is given in~\cite{Johnson}. Using the same ideas as in~\cite{Maclachlan}, Szczepanski~\cite{Szcz} proved the following result: if $r$ is even and $m \ge r$ is odd, then the Fibonacci group $F(r,m)$ can not be the fundamental group of hyperbolic 3-manifold of finite volume. 

\subsection{Sieradski groups} \label{subsection2.3}

One more interesting class of cyclically presented groups corresponds to the defining word $w(x_{1}, x_{2}, x_{3}) = x_{1} x_{3} x_{2}^{-1}$. Cyclically presented groups 
$$
S(m) = \langle x_1, x_2, \ldots , x_{m} \, \mid \,  x_{i} x_{i+2} = x_{i + 1}, \quad i=1, \ldots m \rangle, 
$$
where all subscripts are taken by $\operatorname{mod} m$, were introduced by Sieradski~\cite{S}. In~\cite{CavicchioliHegenbarthKim} these groups were called \emph{the Sieradski groups} as well as a more general class of groups was introduced. The groups  
$$
\begin{gathered}
S(m,p,q) = \langle x_1, \ldots , x_m \, \mid \,  \hfill \\ 
{x_i} \, x_{i + q} \, \cdots \, x_{i + (q - 1)d q - q} \,  x_{i + (q - 1)d q} \,  = \, x_{i + 1} \, x_{i+q+1} \,  \cdots \,  x_{i + (q - 1)d q - q + 1} , \\ \hfill   
 i=1, \ldots, m \rangle ,   
\end{gathered}
$$
are called  \emph{the generalised Sieradski groups}. As above,  all subscripts are taken by $\operatorname{mod} m$. Parameters $p$ and $q$ are co-prime integers such that  $p = 1 + d q$, $d \in \mathbb{Z}$. Cavicchioli, Hegenbarth and Kim proved the following result. 
 
\begin{theorem} \cite{CavicchioliHegenbarthKim} \label{theorem1} 
The cyclic presentation $S(m,p,q)$ corresponds to a spine of the $m$-fold cyclic covering of the 3-sphere $S^{3}$ branched over the torus knot $T(p,q)$, i.e., the Brieskorn manifold $\mathcal B(m,p,q)$. 
\end{theorem}

In particular, the Sieradski groups $S(m) = S(m,3,2)$ correspond to spines of manifolds $\mathcal B(m,3,2)$ which are  $m$-fold cyclic coverings of $S^{3}$ branched over the trefoil knot $T(3,2)$. 

Below we will interested in generalised Sieradski groups with parameter $q=2$. In this case $p=1 + 2d$ and we get the following presentation:  
$$
\begin{gathered} 
S(m,2d+1,2) = \langle x_{1}, x_{2}, \ldots, x_{m} \ | \   x_{i} x_{i+2} \cdots x_{i+2d} \ = \ x_{i+1} x_{i+3} \cdots x_{i+2d-1}, \\ \hfill \quad i=1, \ldots, m \rangle.
\end{gathered}
$$

\subsection{Generalised Fibonacci groups}

Consider two families of groups which were introduced in~\cite{C-R} and called \emph{the generalised Fibonacci groups}. The first family consists of the groups $F(r,m,k)$, where $r \ge 2$, $m \ge 3$, $k \ge 1$, given by   
$$
F(r,m,k)  =  \langle x_1, \dots , x_m \ | \ x_i x_{i+1} \cdots x_{i+r-1} =  x_{i+r-1+k}, \quad i = 1, \dots , m \rangle ,
$$
where all subscripts are taken by $\operatorname{mod} m$. Obviously, $F(r,m,1)=F(r,m)$. 

The following statement generalises preceding results on groups $F(2,m)$ and $F(r,m)$, obtained in~\cite{Maclachlan} and~\cite{Szcz}, respectively. 

\begin{theorem} \cite{Sz-Ves3} \label{theorem2}
Assume that  $r$ is even,  $m$ is odd and $(m, r+2k-1)=1$. Then the generalised Fibonacci group $F(r,m,k)$ can not be realised as the group of hyperbolic 3-orbifold of finite volume. 
\end{theorem}

Let us discuss some groups $F(r,m,k)$ which do not satisfy to conditions of theorem~\ref{theorem2}. Suppose 
\begin{equation}
m = r + 2k - 1 . \label{eqn_cyclic4}
\end{equation}

For $k=1$ in  (\ref{eqn_cyclic4})  we have $m=r+1$. Corresponding groups $F(m-1,m,1)$ are the Fibonacci groups  $F(m-1,m)$ with the following cyclic presentation: 
$$
F(m-1,m) = \langle x_{1}, x_{2}, \ldots, x_{m} \ | \ x_{i} x_{i+1} \cdots x_{i+m-2} = x_{i+m-1}, \quad i=1,  \ldots, m \rangle. 
$$
Recall the definition of \emph{the generalised Neuwirth groups} $\Gamma^{\ell}_m$ from~\cite{Sz-Ves3}:
$$
\Gamma^{\ell}_m = \langle x_1, \dots , x_m \ | \ x_i x_{i+1} \ldots x_{i+m-2} = x_{i+m-1}^{\ell},  \quad  i=1, \ldots,  m \rangle ,
$$
where $m \geq 3$ and $\ell \geq 1$. These groups generalise the groups introduced by Nuewirth in~\cite{N}. In particular, $\Gamma^1_m = F(m-1, m)$. 

\begin{proposition} \cite{Sz-Ves3} 
The groups $\Gamma^{\ell}_m$ are the fundamental groups of the fibered Seifert spaces $\Sigma_{m}^{\ell}$ with parameters 
$$
\Sigma_m^{\ell} \ = \ ( 0 \ {\small o} \ 0 \ | \ -1 ; \ \underbrace{(\ell+1,1), \ (\ell+1,1) , \ \dots , (\ell+1,1)}_{m~{\rm times}} ) .
$$
\end{proposition}

For $r=2$ in (\ref{eqn_cyclic4}) we have  
$$
F(2,2k+1,k) = \langle x_1, \dots , x_{2k+1} \ | \ x_i x_{i+1} =  x_{i+1+k}, \quad i=1, \dots , 2k+1 \rangle .
$$
It is easy to see~\cite{Sz-Ves2} that the groups $F(2,2k+1,k)$ and the Sieradski groups 
$$
S(2k+1) =  \langle a_{1}, a_{2}, \ldots, a_{2k+1} \, | \, a_{i} a_{i+2} = a_{i+1}, \quad i = 1, \ldots, 2k+1 \rangle  
$$ 
are isomorphic under the following correspondence of generators: 
$$
\left(
\begin{array}{llllllllll}
x_1 & x_2 & x_3 & \dots & x_k & x_{k+1} & x_{k+2} & \dots & x_{2k} & x_{2k+1} \\
a_1 & a_3 & a_5 & \dots & a_{2k-1} & a_{2k+1} & a_2 & \dots &
a_{2k-2} & a_{2k}
\end{array}
\right) .
$$
Theorem~\ref{theorem1} implies that the groups $F(2,2k+1,k)$ are the fundamental groups of 3-manifolds which are $(k+1)$-fold cyclic coverings of $S^{3}$ branched over the trefoil knot. 
 
The second family from~\cite{C-R} consists of the groups $H(r,m,k)$, where $r \ge 2$, $m \ge 3$, $k \ge 1$, given by  
$$
H(r,m,k)  =  \langle x_1, \dots , x_m \ | \ x_i x_{i+1} \cdots x_{i+r-1}  = x_{i+r} \cdots x_{i+r-1+k},   \quad i = 1, \dots , m \rangle ,
$$
where all subscripts are taken by $\operatorname{mod} m$. In particular, $H(r,m,1) = F(r,m)$. Let us compare the groups $H(r,m,k)$ with the generalised Sieradski groups $S(m,2k-1,2)$: 
$$
\begin{gathered} 
S(m,2k-1,2) = \langle a_1, \dots , a_m \ | \ a_i a_{i+2} \ \dots \ a_{i+2k-2} = a_{i+1} a_{i+3} \ \dots \ a_{i+2k-3}, \\ \hfill \quad i=1 \dots n \rangle . 
\end{gathered}
$$
It is easy to see~\cite{Sz-Ves2} that $H(k,2k-1,k-1)$ and $S(2k-1,2k-1,2)$ are isomorphic under the following correspondence of generators: 
$$
\left(
\begin{array}{lllllllll}
x_1 & x_2 & x_3 & \dots & x_k & x_{k+1} & x_{k+2} & \dots & x_{2k-1} \\
a_1 & a_3 & a_5 & \dots & a_{2k-1} & a_{2} & a_4 & \dots & a_{2k-2}
\end{array}
\right) .
$$
Theorem~\ref{theorem1} implies that for $k \ge 2$ the group $H(k,2k-1,k-1) \cong S(2k-1,2k-1,2)$ is the fundamental group of the closed 3-manifold which is the $(2k-1)$-fold cyclic covering of the 3-sphere branched over the torus knot  $T(2k-1,2)$, i.e.,  the Brieskorn manifold $\mathcal B(2k-1, 2k-1, 2)$. 

\subsection{Johnson~-- Mawdesley groups}

In~\cite{J-M} Johnson and Mawdesley defined the class of groups with the following cyclic presentations:  
$$
G_{n}(m,k) = \langle x_{1}, x_{2}, \ldots, x_{n} \, | \, x_{i} x_{i+m} = x_{i+k}, \quad  i = 1, \ldots, n \rangle.
$$
Obviously, this class contains the Fibonacci groups and the Sieradski groups, namely $G_{n}(1,2) = F(2,n)$ and $G_{n} (2,1) = S(n,3,2)$. Structure of the groups $G_{n}(m,k)$ was studied in \cite{BardakovVesnin} and properties of the groups $G_{n}(m,1)$ were considered in \cite{G-H}.  In \cite{Howie} there was investigated the following question: when  $G_{n}(m,k)$ is the fundamental group of a 3-manifold? The answer was done for all group with only two exceptional cases: for the groups  $G_{9}(4,1)$ and $G_{9}(7,1)$ the question is still open.   

\section{Geometric group presentations} \label{sec3}

\subsection{Cyclic presentations of $S(2n,3,2)$ with $n$ generators}

It was men\-tioned above that the Fibonacci group $F(2,2n)$, $n \geq 2$, is the fundamental group of a 3-manifold which is the $n$-fold cyclic cover of the 3-sphere $S^{3}$ branched over the figure-eight knot. This manifold is spherical if $n=2$, Euclidean if $n=3$, and hyperbolic if $n \geq 4$. The covering is related to symmetry of order $n$ corresponding to the permutation $x_{i} \to x_{i+2}$ on the generators $x_{1}, x_{2}, \ldots, x_{2n}$ of $F(2,2n)$, but not to the permutation  $x_{i} \to x_{i+1}$, which corresponds to the cyclic presentation of the group. It was observed in~\cite{Kim-Vesnin} that the $n$-fold cyclic covering corresponds to the following $n$-cyclic presentation: 
$$
\begin{gathered} 
F(2,2n) = G_{n}(y_{1}^{-1} y_{2}^{2} y_{3}^{-1} y_{2}) = \langle y_{1}, y_{2}, \ldots y_{n} \ | \ y_{i}^{-1} y_{i+1}^{2} y_{i+2}^{-1} y_{i+1} =1, 
\ i=1, \ldots, n \rangle, 
\end{gathered}
$$
where $y_{i} = x_{2i}$, $i=1, \ldots, n$.  The correspondence $\sum_{i} y_{i}^{k_{i}} \to \sum_{i} k_{i} t^{i}$ sent the defining word  $y_{0}^{-1} y_{1}^{2} y_{2}^{-1} y_{1}$ to the polynomial $-(t^{2} - 3t + 1)$, which is equivalent to the Alexander polynomial $\Delta(t) = t^{2} - 3t + 1$ of the figure-eight knot.  

Analogously, one can consider the generalised Sieradski group $S(2n,3,2)$ with even number of generators.  This group admits a cyclic presentation with $n$ gene\-ra\-tors:  
\begin{eqnarray*}
S(2n,3,2) & = & G_{2n} (x_{1} x_{3} x_{2}^{-1}) \\
& = & \langle x_{1}, x_{2}, \ldots, x_{2n} \ | \ x_{i} x_{i+2} = x_{i+1}, \quad i=1, \ldots, 2n \rangle \\ 
& = & \langle x_{1}, x_{2}, \ldots, x_{2n} \ | \ x_{2j} x_{2j+2} = x_{2j+1}, \quad x_{2j+1} x_{2j+3} = x_{2j+2}, \\ 
& & \qquad \qquad \qquad \qquad  \qquad \qquad j=1, \ldots, n \rangle \\ 
& = & \langle x_{2}, x_{4}, \ldots, x_{2n} \ | \ (x_{2j} x_{2j+2}) (x_{2j+2} x_{2j+4}) = x_{2j+2},  \quad j=1, \ldots, n \rangle \\ 
& = & \langle y_{1}, y_{2}, \ldots, y_{n} \ | \ y_{j} y_{j+1}^{2} y_{j+2} = y_{j+1}, \quad j = 1, \ldots, n \rangle \\ 
& = & G_{n} (y_{1} y_{2}^{2} y_{3} y_{2}^{-1}).  
\end{eqnarray*}
Theorem~\ref{theorem1} implies that $S(2n,3,2)$ is the fundamental group of the 3-manifold $\mathcal B(2n,3,2)$, which is the $2n$-fold cyclic covering of $S^{3}$ branched over the trefoil knot. Moreover, the cyclic presentation $S(2n,3,2)$ is geometric. 

It is natural to ask: if the cyclic presentation $G_{n} (x_{1} x_{2}^{2} x_{3} x_{2}^{-1})$ is geometric too? 

Answering in affirmative on the question, Howie and Williams constructed a corresponding spine in~\cite{Howie}. Following notations from~\cite{Howie}, we denote the ge\-ne\-ra\-tors by $x_{0}, x_{1}, \ldots, x_{n-1}$. Consider the 2-complex $\mathcal P_{n}$, presented in figure~\ref{fig1}. 
\begin{figure}[h]
\centering{
\includegraphics[totalheight=6.cm]{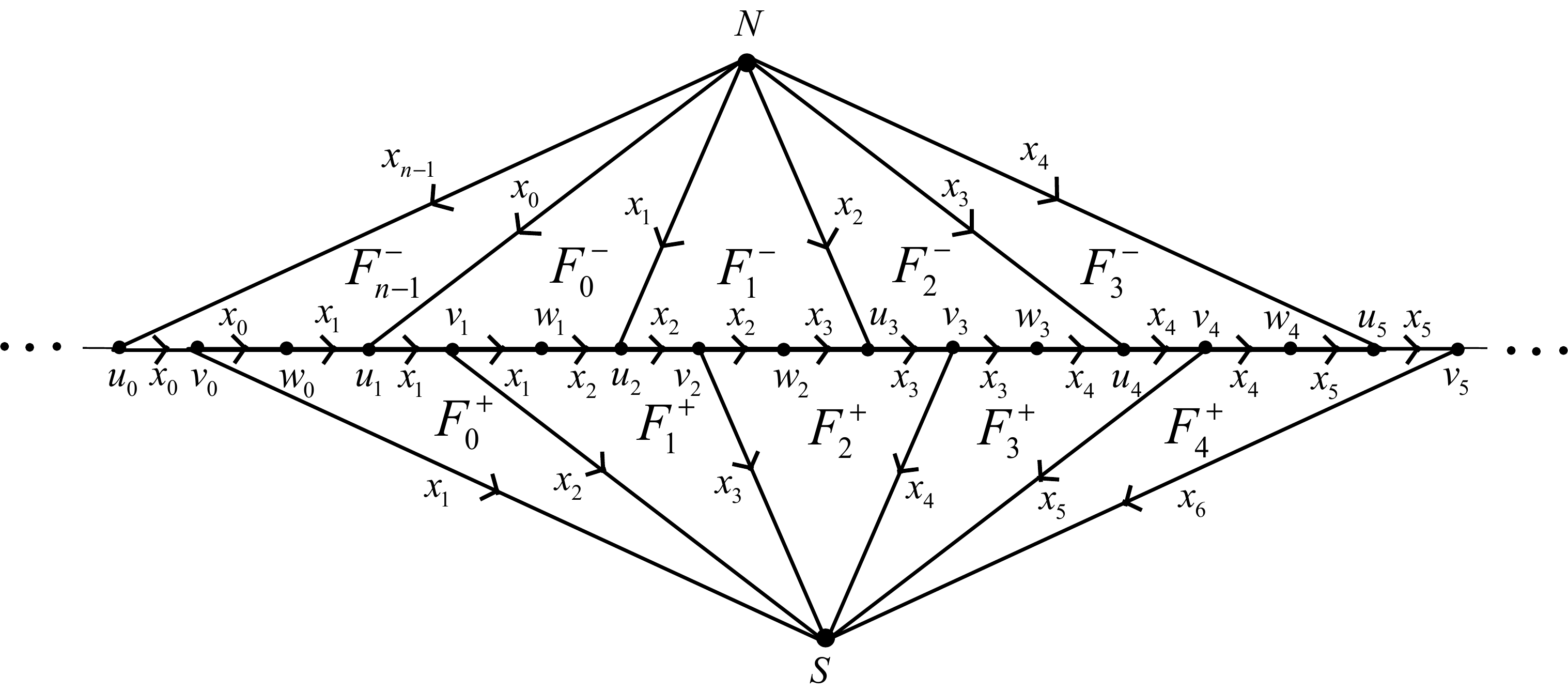}
\caption{The complex $\mathcal P_{n}$.} \label{fig1}
}
\end{figure}
It has $2n$ pentagonal 2-cells denoted by $F^{+}_{i}$, $F^{-}_{i}$, where $i=0, \ldots, n-1$. Let us orient edges and mark them by $x_{0}, x_{2}, \ldots, x_{n-1}$ as in the figure, and denote vertices by $N$, $S$, $u_{i}$, $v_{i}$, $w_{i}$ for $i=0, \ldots, n-1$.  

Let us define pairing $F_{i}$ of the 2-cells of  $\mathcal P_{n}$ in such a way that for each $i=0, \ldots, n-1$ the cells  $F^{-}_{i}$ and $F^{+}_{i}$ are identified according to the following order of vertices:   
$$
F_{i} : F_i^-  = (N u_{i+1} v_{i+1} w_{i+1} u_{i+2})  \longrightarrow  F_i^+ =  (v_{i} w_{i} u_{i+1} v_{i+1} S). 
$$
The pairing of 2-cells induces splitting of edges in classes of equivalent. In particular, for the complex, presented in figure~\ref{fig1}, the class of edges, marked by  ${x_3}$, consists of the following edges, equivalent under pairings $F_{i}$ and their inverses:  
$$
x_3 \, : \, [N, u_{4}] \stackrel{F_2}{\to} [v_{2}, S]  \stackrel{F_1^{-1}}{\to}  [w_{2}, u_{3}] \stackrel{F_2^{-1}}{\to}  [u_{3}, v_{3}] \stackrel{F_{2}^{-1}}{\to}  [v_{3}, w_{3}] \stackrel{F_{3}^{-1}}{\to}  [N, u_{4}], 
$$ 
where edges are presented by initial and terminal vertices and labels of arrows indicate gluings. 
  
The following property holds. 

\begin{theorem} \cite{Howie}  \label{teo3}
Cyclic presentation $G_{n} (x_{0} x_{1}^{2} x_{2} x_{1}^{-1})$ is geometric, i.e., corres\-ponds to a spine of a closed 3-manifold.  
\end{theorem}

Spine $\mathcal P_{n}$ has a cyclic symmetry of order $n$ which corresponds to the cyclic presenta\-tion $G_{n} (x_{0} x_{1}^{2} x_{2} x_{1}^{-1})$. This symmetry induces an order $n$ cyclic symmetry of the manifold. We will describe the corresponding quotient space. 

\begin{proposition} \label{ass1}
For any $n$ the manifold from theorem~\ref{teo3} is an $n$-fold cyclic branched covering of the lens space~$L(3,1)$.
\end{proposition}

\begin{proof} 
Let $\mathcal S(2n,3,2)$ be the closed 3-manifold, constructed from the spine $P_{n}$. To prove the statement  we will use the standard way to move from a spine to a Heegaard diagram of the manifold. As the result, we get the Heegaard diagram, presented in figure~\ref{fig2}.  
\begin{figure}[h]
\begin{center}
\includegraphics[totalheight=4.cm]{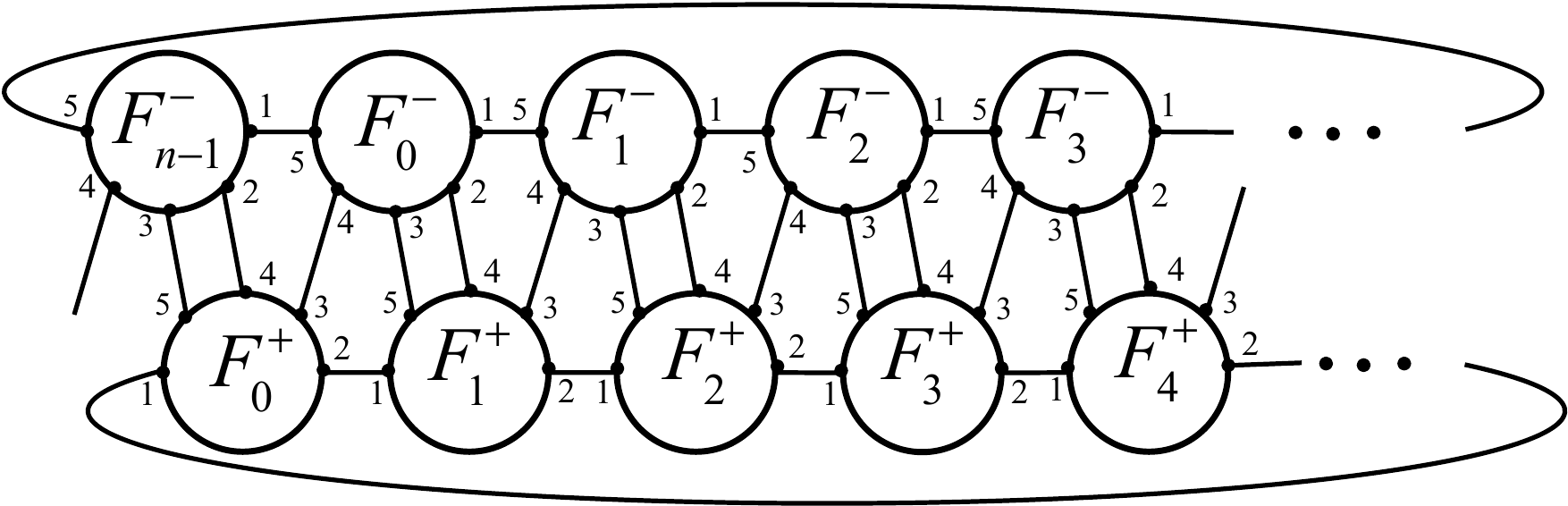}
\caption{Heegaard diagram of $\mathcal S(2n,3,2)$.}\label{fig2}
\end{center}
\end{figure}
This diagram admits a rotational symmetry of order $n$ which permutes discs cyclically: $F_{i}^{-} \to F_{i+1}^{-}$ and $F_{i}^{+} \to F_{i+1}^{+}$. Denote the symmetry by $\rho$ and its axe by $\ell$. The quotient space, corresponding to $\rho$, is a 3-orbifold and the image of $\ell$ is a singular set of it. By Singer's move of type IB \cite{Singer} we get the standard Heegaard diagram of the lens space $L(3,1)$, see figure~\ref{zin}. Thus, the underlying space of the 3-orbifold $\mathcal S(2n,3,2)  /  \rho$ is the lens space $L(3,1)$.
\begin{figure}[h]
\begin{center}
\includegraphics[totalheight=5.cm]{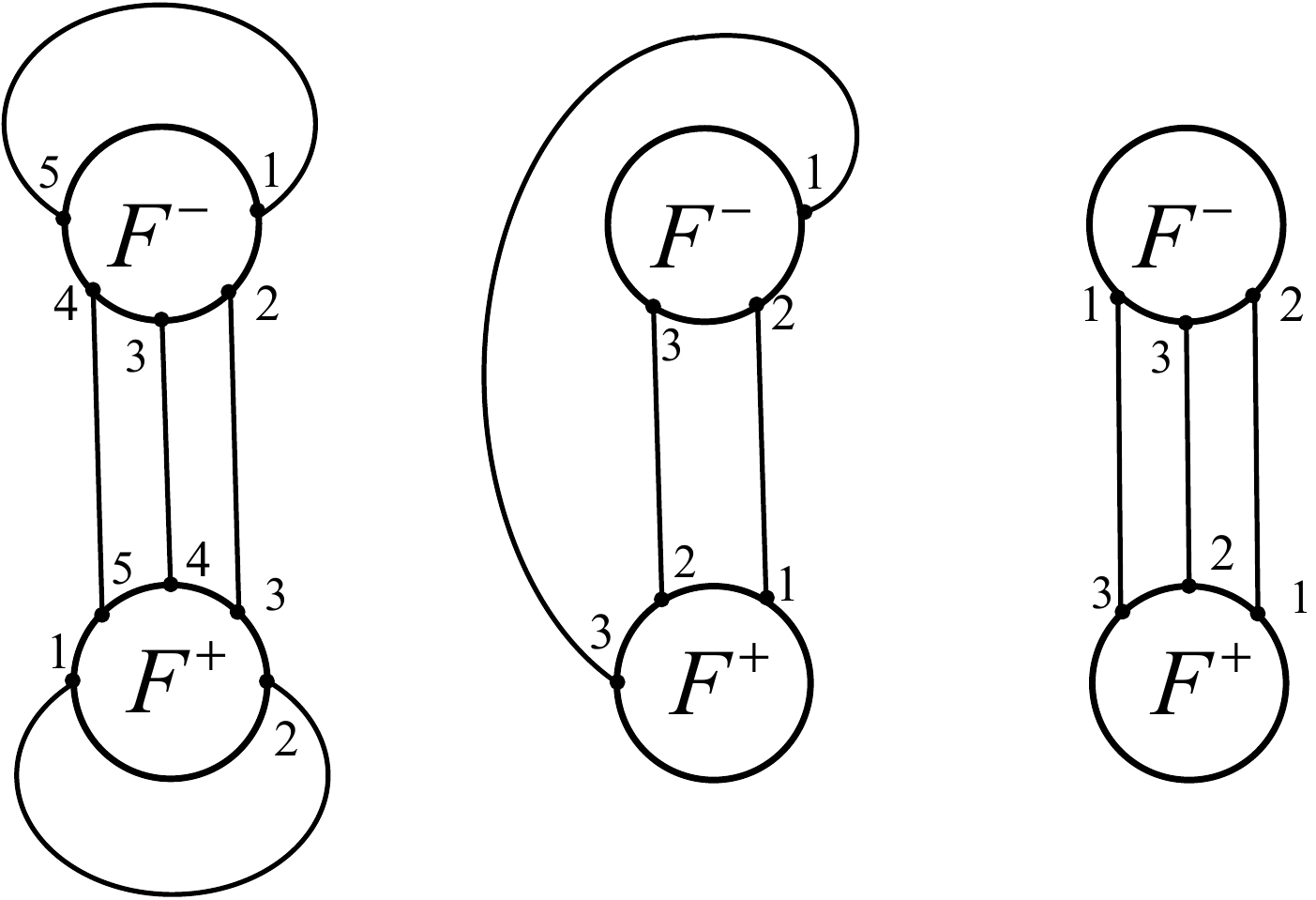} 
\caption{Simplifying a Heegaard diagram of $\mathcal S(2n,3,2)  /  \rho$.} \label{zin}
\end{center}
\end{figure}
Therefore, the manifold $\mathcal S(2n,3,2)$ is an $n$-fold cyclic branched covering of the lens space $L(3,1)$.  
\end{proof}

\subsection{Cyclic presentations of $S(2n,5,2)$ with $n$ generators}

Now we consider the generalised Sieradski groups $S(2n,5,2)$ with even number of generators. From cyclic presentation with $2n$ generators we move to cyclic presentation with $n$ generators:  
 \begin{eqnarray*}
S(2n,5,2) & = & G_{2n} (x_{1} x_{3} x_{5} x_{4}^{-1} x_{2}^{-1}) \\
& = & \langle x_{1}, x_{2}, \ldots, x_{2n} \, | \, x_{i+1} x_{i+3} x_{i+5} = x_{i+2} x_{i+4}, \quad i=1, \ldots, 2n \rangle \\ 
& = & \langle x_{1}, x_{2}, \ldots, x_{2n} \, | \, x_{j+1} x_{j+3} x_{j+5} = x_{j+2} x_{j+4},  \quad  \\ 
& & \qquad  \qquad  \qquad \quad x_{j+2} x_{j+4} x_{j+6} = x_{j+3} x_{j+5}, \quad j=2,4 \ldots, 2n \rangle \\ 
& = & \langle x_{1}, x_{2}, \ldots, x_{2n} \, | \, x_{j+5} = (x_{j} x_{j+2} x_{j+4})^{-1} x_{j+2} x_{j+4},  \quad  \\ 
& & \qquad  \qquad  \qquad \quad x_{j} x_{j+2} x_{j+4} = x_{j+1} x_{j+3}, \quad j=2,4 \ldots, 2n \rangle \\ 
& = & \langle x_{2}, x_{4}, \ldots, x_{2n} \, | \, x_{j} x_{j+2} x_{j+4} \left( x_{j+2}^{-1} x_{j}^{-1} x_{j-2} x_{j} x_{j+2} \right) \\
& &  \qquad \qquad  \qquad \cdot \left( x_{j}^{-1} x_{j-2}^{-1} x_{j-4} x_{j-2} x_{j} \right) = 1,  \quad  j=2, 4 \ldots, 2n \rangle \\ 
& = & \langle y_{1}, y_{2}, \ldots, y_{n} \, | \,  y_{i} y_{i+1} y_{i+2} y_{i+1}^{-1} y_{i}^{-1} y_{i-1} y_{i} y_{i+1} y_{i}^{-1} y_{i-1}^{-1} y_{i-2} y_{i-1} y_{i} =1, \\ 
& & \qquad \qquad i = 1, \ldots, n \rangle \\ 
& = & G_{n} (y_{3} y_{4} y_{5} y_{4}^{-1} y_{3}^{-1} y_{2} y_{3} y_{4} y_{3}^{-1} y_{2}^{-1} y_{1} y_{2} y_{3}  ) \\  
& = & G_{n} (y_{1} y_{2} y_{3} y_{3} y_{4} y_{5} y_{4}^{-1} y_{3}^{-1} y_{2} y_{3} y_{4} y_{3}^{-1} y_{2}^{-1} ).
\end{eqnarray*}
We will show that obtained cyclic presentation is geometric. To formulate the subsequence statement we will renumerate  generators similar to theorem~\ref{teo3}, namely, we put $x_i = y_{i+1}$, where $i = 0,1, \ldots, n-1$. 

\begin{theorem}  \label{teo4}
The cyclic presentation $G_{n} ( x_{0} x_{1} x_{2} x_{2} x_{3} x_{4} x_{3}^{-1} x_{2}^{-1} x_{1} x_{2} x_{3} x_{2}^{-1} x_{1}^{-1})$ is geometric, i.e., it corresponds to a spine of a closed 3-manifold.    
\end{theorem}

\begin{proof} To prove the statement we will construct a $2$-complex $Q_{n}$ with $2n$ 2-cells and oriented edges marked by $x_{i}^{\pm1}$.  Each 2-cell will be an 13-gon with the property that reading marks along its boundary gives some defining relation of the group. Moreover, for each defining relation there are exactly two 2-cell and their orientations are opposite. 

Since the defining relations are quite long, we will demonstrate the constructi\-on for a small example. Let us consider the 2-complex $Q_{4}$, presented in figure~\ref{new}. $Q_{4}$ has eight 2-cells. Looking at the figure we keep in mind that edges on the left side should be  identified in pairs with edges on the right side. And, also, that all vertical lines going upstairs meet in one point, as well as all vertical lines going downstairs meet in one point. The complex $Q_{4}$ corresponds to the cyclic presentation for $n=4$. Indeed, in this case we have four generators, say $x_{0}$, $x_{1}$, $x_{2}$ and $x_{3}$, and four defining relations, which we can write in the following form: 
$$
\begin{gathered}
x_{0} x_{1} x_{2} x_{2} x_{3} x_{0} x_{3}^{-1} x_{2}^{-1} x_{1} x_{2} x_{3} x_{2}^{-1} x_{1}^{-1} = 1, \cr 
x_{1} x_{2} x_{3} x_{3} x_{0} x_{1} x_{0}^{-1} x_{3}^{-1} x_{2} x_{3} x_{0} x_{3}^{-1} x_{2}^{-1} = 1, \cr 
x_{2} x_{3} x_{0} x_{0} x_{1} x_{2} x_{1}^{-1} x_{0}^{-1} x_{3} x_{0} x_{1} x_{0}^{-1} x_{3}^{-1} = 1, \cr 
x_{3} x_{0} x_{1} x_{1} x_{2} x_{3} x_{2}^{-1} x_{1}^{-1} x_{0} x_{1} x_{2} x_{1}^{-1} x_{0}^{-1} = 1. \cr     
\end{gathered} 
$$
It is easy to check that all four words in these relations appear in the process of reading marks along boundaries of the 2-cell in figure~\ref{new}. Indeed, we will get the first word by reading along the boundary of the 2-cell $F_{1}$ in the counterclockwise direction, as well as along  the boundary of the 2-cell  $\overline{F}_{1}$ in the clockwise direction. Analogously, other words correspond to pairs $F_{2}$ and $\overline{F}_{2}$, $F_{3}$ and $\overline{F}_{3}$,  $F_{4}$ and $\overline{F}_{4}$, respectively.  
\begin{figure}[h]
\begin{center}
\includegraphics[totalheight=6.0cm]{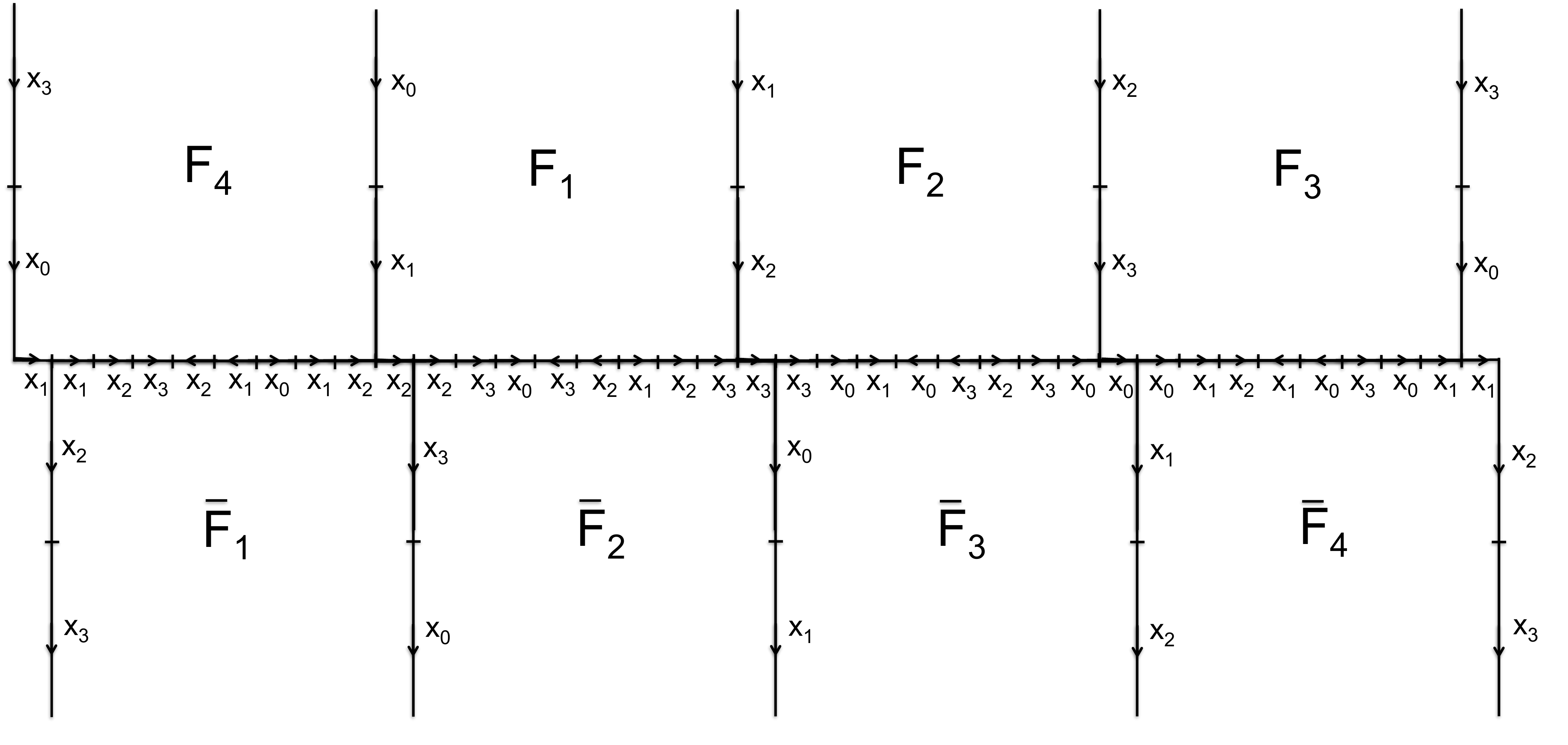} 
\caption{The complex $\mathcal Q_{4}$.} \label{new}
\end{center}
\end{figure}

Let us identify pairs of the 2-cells of $Q_{4}$ with equal words by following marks on their boundaries. These pairings induce splittings of 1-cells and 0-cells in classes of equivalent. The resulting quotient space is a closed 3-dimensional pseudo-manifold. It is ealy to check directly that its Euler characteristic vanishes. Therefore, complex $Q_{4}$ is a spine of a closed orientable 3-manifold~\cite{Seifert}.  The described construction and above arguments work in the same way for arbitrary $n$.   
\end{proof}

Let us denote by $\mathcal S(2n,5,2)$ the closed orientable 3-manifolds, corresponding to the spine $Q_{n}$. For small $n$ manifolds $\mathcal S(2n,5,2)$ can be distinguished by using \emph{3-manifold recognizer}~\cite{Recognizer}. For example, the computations for the case $n=4$ give us that $\mathcal S(8,5,2)$ is the Seifert manifold $(S^2,(4,1),(5,2),(5,2),(1,-1) )$. 

Analogously to proposition~\ref{ass1}, the following result holds. 

\begin{proposition} \label{ass2}
For each $n$ the manifold from theorem~\ref{teo4} is an $n$-fold cyclic covering of the lens space $L(5,1)$.
\end{proposition}

\begin{proof}
Because of the cyclic symmetry, it is enough to consider an example with small parameters. Consider the Heegaard diagram of the manifold $\mathcal S(8,5,2)$ which is obtained from the complex $Q_{4}$, see figure~\ref{fig5}. By the construction, the diagram has a rotational symmetry $\rho$ of order $4$. 
\begin{figure}[h]
\begin{center}
\includegraphics[totalheight=4.cm]{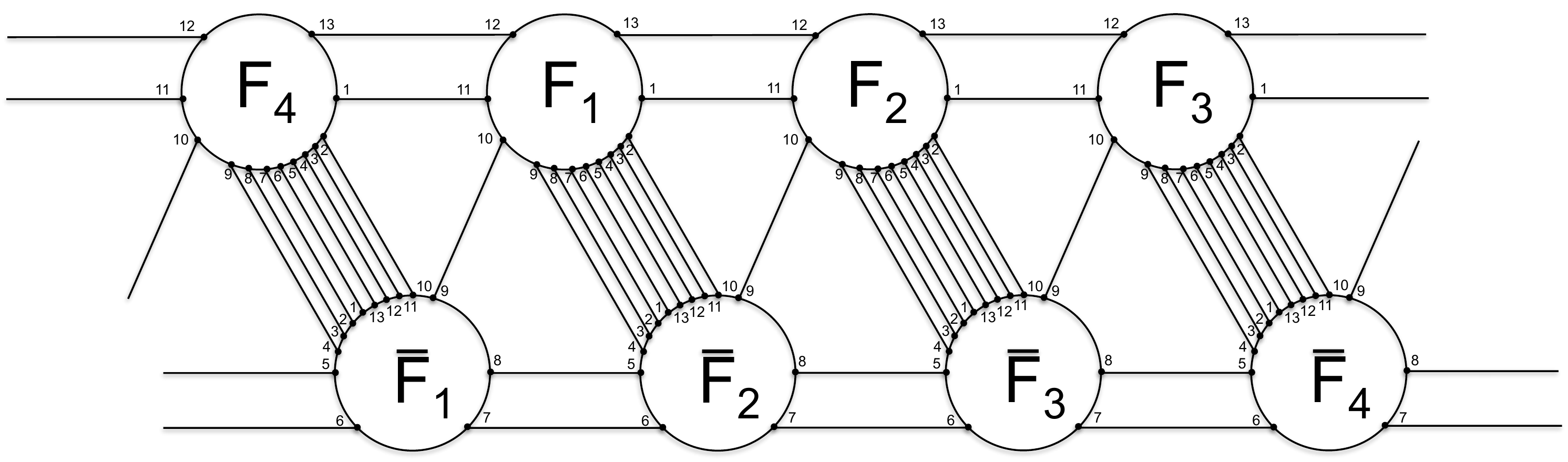} 
\caption{Heegaard diagram for the case $n=4$.} \label{fig5}
\end{center} 
\end{figure}
Analogously to the proof of proposition~\ref{ass1}, it is easy to see that the quotient space $\mathcal S(8,5,2) / \rho$ is the 3-orbifold whose underlying space is the lens space $L(5,1)$. 
\end{proof} 


\end{document}